\documentclass{amsart}
\usepackage{amsmath,amsthm,amssymb,amscd,setspace,mathrsfs,color}

\usepackage{graphicx}

\usepackage[all]{xy}

\newtheorem{thm}{Theorem}[section]
\newtheorem{prop}[thm]{Proposition}
\newtheorem{lem}[thm]{Lemma}
\newtheorem{cor}[thm]{Corollary}
\newtheorem{conj}[thm]{Conjecture}
\newtheorem{prob}[thm]{Problem}

\theoremstyle{definition}
\newtheorem{dfn}[thm]{Definition}
\theoremstyle{remark}

\theoremstyle{remark}
\newtheorem{rem}[thm]{Remark}

\newcommand{\ZZ}{\mathbb{Z}}

\newcommand{\Hom}{{\rm Hom}}

\newcommand{\K}{K}

\renewcommand{\L}{L}

\newcommand{\X}{X}

\newcommand{\Sd}{{\rm Sd}}
\newcommand{\st}{\mathsf{st}^\circ}
\newcommand{\St}{\overline{\mathsf{st}}}

\newcommand{\HH}{\boldsymbol{H}}

\newcommand{\conn}{{\rm conn}}

\newcommand{\GG}{\mathscr{G}}
\renewcommand{\SS}{\mathscr{S}}

\newcommand{\ind}{{\rm ind}}

\title{$\ZZ_2$-indices and Hedetniemi's conjecture}
\author{Takahiro Matsushita}
\address{Department of Mathematical Sciences, University of the Ryukyus,
Nishihara-cho, Okinawa, 903-0213, Japan}

\email{mtst@sci.u-ryukyu.ac.jp}

\subjclass[2010]{Primary 55U10, 05C15, Secondary 55P91}

\begin{document}

\maketitle

\begin{abstract}
The $\ZZ_2$-index $\ind(X)$ of a $\ZZ_2$-CW-complex $X$ is the smallest number $n$ such that there is a $\ZZ_2$-map from $X$ to $S^n$. Here we consider $S^n$ as a $\ZZ_2$-space by the antipodal map. Hedetniemi's conjecture is a long standing conjecture in graph theory concerning the graph coloring problem of tensor products of finite graphs. We show that if Hedetniemi's conjecture is true, then $\ind(X \times Y) = \min \{ \ind(X) , \ind(Y)\}$ for every pair $X$ and $Y$ of finite $\ZZ_2$-complexes.
\end{abstract}

\section{Introduction}

The {\it $\ZZ_2$-index $\ind(X)$ of a $\ZZ_2$-CW-complex $X$} is the smallest integer $n$ such that there is a $\ZZ_2$-map from $X$ to $S^n$. Here we consider $S^n$ as a $\ZZ_2$-space by the antipodal map. This invariant has been studied in homotopy theory for a long time (see \cite{CF} and \cite{Stolz}). On the other hand, Hedetniemi's conjecture is a long standing conjecture in graph theory, concerning the chromatic numbers of tensor products of graphs (see \cite{Sauer}, \cite{Tardif}, and \cite{Zhu}). The purpose of this paper is to correspond Hedetniemi's conjecture to $\ZZ_2$-indices of product spaces.

We first recall some basic terminology in graph theory. An {\it $n$-coloring of a simple graph $G$} is a function $c : V(G) \rightarrow \{ 1,\cdots, n\}$ such that $(v,w) \in E(G)$ implies $c(v) \neq c(w)$. The {\it chromatic number $\chi(G)$ of $G$} is the smallest integer $n$ such that $G$ has an $n$-coloring. The graph coloring problem, which is one of the most classical problems in graph theory, is to determine the chromatic number.

The following conjecture was suggested by Hedetniemi \cite{Hedetniemi} in 1966. Although there is not much supporting evidence, this conjecture has still survived.

\begin{conj}[Hedetniemi's conjecture] \label{conj 1.1}
For every pair of finite graphs $G$ and $H$, the equality
$$\chi(G \times H) = \min \{ \chi(G), \chi(H)\}$$
holds. Here we denote by $G \times H$ the tensor product (see Section 2) of $G$ and $H$.
\end{conj}

Following \cite{Tardif}, we say that the proposition {\it $\HH(n)$ holds} if and only if for every pair $G$ and $H$ of finite graphs, $\chi(G \times H) = n$ implies $\min \{ \chi(G), \chi(H)\} = n$. Thus Hedetniemi's conjecture is true if and only if $\HH(n)$ holds for every $n$. Similarly, let us say that {\it $\HH' (n)$ holds} if for every pair $X$ and $Y$ of finite $\ZZ_2$-complexes, $\ind(X \times Y) = n$ implies $\min \{ \ind(X), \ind(Y)\} = n$. In this terminology, our main result is formulated as follows:

\begin{thm}\label{thm 1.2}
For a positive integer $n$, $\HH(n)$ implies $\HH'(n-2)$. Therefore if Hedetniemi's conjecture is true, then the equality
$$\ind(X \times Y) = \min \{ \ind(X), \ind(Y)\}$$
holds for every pair $X$ and $Y$ of finite $\ZZ_2$-complexes.
\end{thm}

Hence, if there are finite $\ZZ_2$-complexes $X$ and $Y$ such that $\ind(X \times Y)$ is smaller than $\min \{ \ind(X), \ind(Y)\}$, then Hedetniemi's conjecture is false. Since it is easy to see that $\ind(X \times Y) \le \min \{\ind(X), \ind(Y)\}$, we pose the following problem:

\begin{prob}
Does there exist a pair of finite $\ZZ_2$-complexes $X$ and $Y$ such that $\ind(X \times Y) < \min \{ \ind(X) , \ind(Y)\}$?
\end{prob}

It is clear that $\HH(1)$ and $\HH(2)$ are true. El-Zahar and Sauer \cite{ES} show that $\HH(3)$ is true. Thus Theorem \ref{thm 1.2} gives the following corollary. However, this corollary has a direct topological proof, which will be given in Section 6.

\begin{cor} \label{cor 1.2.1}
Let $X$ and $Y$ be finite $\ZZ_2$-complexes. If there is a $\ZZ_2$-map from $X \times Y$ to $S^1$, then either $X$ or $Y$ has a $\ZZ_2$-map into $S^1$.
\end{cor}

In the proof of Theorem \ref{thm 1.2}, we use the box complex functor. The box complex $B(G)$ is a $\ZZ_2$-space associated to a graph $G$ introduced in the context of topological lower bounds for chromatic numbers (see Section 3). This complex has a long history starting from Lov\'asz's proof of Kneser's conjecture \cite{Lovasz}, and we refer to \cite{Kozlov book} for the background.

In this paper, we use the right adjoint property of the box complex functor. Since the box complex $B$ is a functor, the existence of graph homomorphisms from $G$ to $H$ clearly implies the existence of $\ZZ_2$-simplicial maps from $B(G)$ to $B(H)$. On the other hand, if the box complex $B$ is merely a functor, then there is no reason which associate the existence of $\ZZ_2$-simplicial maps to the existence of graph homomorphisms. However, if $B$ is a right adjoint functor, one can deduce the existence of graph homomorphisms to $G$ from the existence of $\ZZ_2$-simplicial maps to $B(G)$.

Here is a slightly complicated problem. Namely, there are several definitions of box complexes and some of them have no left adjoint. Moreover, in the proof of Theorem \ref{thm 1.2}, we want its left adjoint preserve finite products. So the choice of definitions of box complexes is a crucial problem.

To state our main technical result explicitly, we denote by $\GG$ the category of finite graphs, and by $\SS^{\ZZ_2}$ the category of finite $\ZZ_2$-simplicial complexes. Csorba \cite{Csorba} constructed a functor $A : \SS^{\ZZ_2} \rightarrow \GG$ (the precise definition is found in Section 3) to show that for every free $\ZZ_2$-simplicial complex $K$, there is a graph whose box complex is $\ZZ_2$-homotopy equivalent to $|K|$. In Section 3, we show that this functor has the following remarkable properties.

\begin{thm} \label{thm 1.3}
Csorba's functor $A : \SS^{\ZZ_2} \rightarrow \GG$ is a left adjoint functor preserving finite limits. Moreover, the geometric realization of its right adjoint is naturally $\ZZ_2$-homotopy equivalent to the box complex functor.
\end{thm}

Our box complex $B$ is the left adjoint to $A$. In fact, this box complex is rarely found in the literature, and the precise description is given in Section 3.

\begin{rem} \label{rem 1.6}
Here we mention the converse of Theorem \ref{thm 1.2}. Of course, if Hedetniemi's conjecture is true, then the converse clearly holds. However, it seems to be difficult to prove the converse of Theorem \ref{thm 1.2} directly. As was mentioned, we associate the existence of $\ZZ_2$-simplicial maps to the existence of graph homomorphisms in the proof of this theorem. However, the existence of $\ZZ_2$-simplicial maps is strictly stronger than the existence of $\ZZ_2$-continuous maps.
\end{rem}

Finally, we mention that Theorem \ref{thm 1.2} was independently proved by Wrochna \cite{Wrochna2} in a different way. He used the $k$-th inverse power $\Omega_k$ for a positive odd integer $k$, which has been used in graph theory (see \cite{GJS}, \cite{HH}, \cite{HT}, and \cite{Tardif1}). In fact, $\Omega_k$ is a right adjoint functor (and hence they preserves products), and Wrochna proved some close relationship between $\Omega_k$ and box complexes. Roughly speaking, he showed that $\Omega_k$ plays a role similar to the barycentric subdivision of simplicial complexes. Namely, he showed that $B(\Omega_k(G))$ is naturally $\ZZ_2$-homotopy equivalent to $B(G)$, and that there is a $\ZZ_2$-continuous map from $B(G)$ to $B(H)$ if and only if there is a graph homomorphism $\Omega_k(G) \to H$ for sufficiently large $k$.



The rest of the paper is organized as follows. In Section 2, we briefly review basic facts and definitions concerning graphs and simplicial complexes. In Section 3, we recall the definition of Csorba's functor $A$ and prove Theorem \ref{thm 1.3}. In Section 4, we show the $\ZZ_2$-simplicial approximation theorem for products of $\ZZ_2$-simplicial complexes. We need this theorem to relate the $\ZZ_2$-indices of products to chromatic numbers of tensor products. We complete the proof of Theorem \ref{thm 1.2} in Section 5. Finally, we give a direct proof of Corollary \ref{cor 1.2.1} in Section 6.

\vspace{2mm} \noindent
{\bf Acknowledgements.} The author thanks Daisuke Kishimoto for helpful comments. He also thanks to Marcin Wrochna for informing his work \cite{Wrochna2} which overlaps many parts with this paper. Finally, he thanks to the anonymous referees for their valuable comments, which made the manuscript much more readable. The author was supported by the Grand-in-Aid for Scientific Research (KAKENHI 28-6304).

\section{Preliminaries}

In this section, we review necessary definitions and facts concerning graphs and simplicial complexes. For a more concrete explanation, we refer to \cite{Kozlov book}. Moreover, we need a bit of knowledge of category theory, which is found in \cite{Leinster}.

\subsection{Graphs}
A {\it graph} is a pair $G = (V(G), E(G))$ consisting of a finite set $V(G)$ together with a symmetric subset $E(G)$ of $V(G) \times V(G)$. A graph having no looped vertices is called a {\it simple graph}. For a pair $v$ and $w$ of vertices, we write $v \sim w$ to mean that $v$ and $w$ are adjacent. A graph homomorphism is a map $f: V(G) \rightarrow V(H)$ such that $(v,w) \in E(G)$ implies $(f(v), f(w)) \in E(H)$. Let $\GG$ denote the category of graphs whose morphisms are graph homomorphisms.

For a non-negative integer $n$, the {\it complete graph $K_n$ with $n$-vertices} is defined as follows: The vertex set $V(K_n)$ is the $n$-point set $[n] = \{ 1,\cdots, n\}$, and the edge set $E(K_n)$ is $\{ (i,j) \; | \; i,j \in [n], i \neq j\}$. Then an $n$-coloring of $G$ is identified with a graph homomorphism from $G$ to $K_n$. In particular, if there is a graph homomorphism from $G$ to $H$, then we have $\chi(G) \leq \chi(H)$.

Let $G$ and $H$ be graphs. The {\it tensor (or categorical) product $G \times H$} is defined by
$$V(G \times H) = V(G) \times V(H),$$
$$E(G \times H) = \{ ((v,w),(v',w')) \; | \; (v,v') \in E(G), \; (w,w') \in E(H)\}.$$
Since the projections $G \times H \rightarrow G$ and $G \times H \rightarrow H$ are graph homomorphisms, the inequality $\chi(G \times H) \leq \min \{ \chi(G) , \chi(H)\}$ is obvious.

\subsection{Simplicial complexes}
In this paper, we use the term ``simplicial complex'' to mean ``finite abstract simplicial complex''. The geometric realization of a simplicial complex $K$ is denoted by $|K|$.

The {\it face poset $FK$ of a simplicial complex $K$} is the poset of non-empty simplices ordered by inclusion. The {\it order complex $\Delta(P)$ of a poset $P$} is the simplicial complex consisting of finite chains in $P$. The order complex of the face poset of $K$ is called the {\it barycentric subdivision of $K$}, and is denoted by $\Sd(K)$. It is known that there is a natural homeomorphism $|K| \to |\Sd(K)|$.

A $\ZZ_2$-action on a simplicial complex $K$ is identified with an involution of $K$, i.e. a simplicial map $\alpha : K \rightarrow K$ satisfying $\alpha^2 = {\rm id}_K$. The category of $\ZZ_2$-simplicial complexes whose morphisms are $\ZZ_2$-simplicial maps is denoted by $\SS^{\ZZ_2}$.

For a pair $K$ and $L$ of simplicial complexes, we define the {\it product $K \times L$} as follows: The vertex set of $K \times L$ is $V(K) \times V(L)$. Let $p_1 : V(K) \times V(L) \rightarrow V(K)$ and $p_2 : V(K) \times V(L) \rightarrow V(L)$ be projections. Then $\sigma \subset V(K) \times V(L)$ is a simplex of $K \times L$ if and only if $p_1(\sigma) \in K$ and $p_2(\sigma) \in L$. Note that $|K \times L|$ and $|K| \times |L|$ are not homeomorphic in general. For example, if $K$ and $L$ are 1-simplices, then $|K \times L|$ is a 3-simplex but $|K| \times |L|$ is a square. If $K$ and $L$ are $\ZZ_2$-simplicial complexes with involutions $\alpha$ and $\beta$ respectively, then we consider the involution of $K \times L$ as $\alpha \times \beta$, i.e. the map $(v,w) \mapsto (\alpha(v), \beta(w))$.

\section{Box complexes and its left adjoint}

The purpose of this section is to prove Theorem \ref{thm 1.3}. As was mentioned in Section 1, it is important to consider how to formulate the box complex functor. Thus we first give the definition of our box complexes, and see that it is equivalent to other box complexes.

We start with recalling the exponential graphs, following \cite{Dochtermann} and \cite{ES}. For a pair $G$ and $H$ of graphs, the {\it exponential graph $H^G$} is defined by
$$V(H^G) = \{ f: V(G) \rightarrow V(H) \; | \; \textrm{$f$ is a set map}\},$$
$$E(H^G) = \{ (f,g) \; | \; \textrm{$(v,w) \in E(G)$ implies $(f(v), g(w)) \in E(H)$}\}.$$
Note that a vertex $f$ of $V(H^G)$ is looped if and only if $f$ is a graph homomorphism.

A {\it (looped) clique of a graph $G$} is a subset $\sigma$ of $V(G)$ such that $\sigma \times \sigma \subset E(G)$. The {\it clique complex $C(G)$} is the simplicial complex consisting of cliques of $G$.

\begin{dfn}
The box complex $B(G)$ of a graph $G$ is the clique complex of $G^{K_2}$. Note that a $\ZZ_2$-action on $K_2$, which flips the edge, induces a $\ZZ_2$-action on $B(G)$.
\end{dfn}

Note that a vertex of $B(G)$ is a graph homomorphism $e \colon K_2 \to G$, and hence is identified with an element $(e(1), e(2))$ of $E(G)$. In this way, we sometimes identify the vertex set of $B(G)$ with $E(G)$.

We now prove that our box complex is $\ZZ_2$-homotopy equivalent to other box complexes. We only draw a comparison between our box complex $B(G)$ and $\Hom(K_2,G)$, which is one of the most famous formulations of box complexes. In fact, this is easily deduced from known results. Since in this paper we only use the complex $\Hom(K_2, G)$ to connect $B(G)$ with other box complexes, we do not give the definition, but only refer to \cite{BK1}. For comparisons between $\Hom(K_2,G)$ and other box complexes, we refer to \cite{Zivaljevic}.

\begin{prop} \label{prop 3.2}
Our box complex $B(G)$ is naturally $\ZZ_2$-homotopy equivalent to $\Hom(K_2,G)$.
\end{prop}
\begin{proof}
This follows from Proposition 3.5 and Remark 3.6 in \cite{Dochtermann}.
\end{proof}

\begin{cor} \label{cor 3.2.1}
$B(K_n)$ is $\ZZ_2$-homotopy equivalent to $S^{n-2}$.
\end{cor}
\begin{proof}
This follows from Proposition 4.3 of \cite{BK1} and Proposition \ref{prop 3.2}.
\end{proof}

\begin{rem}
Lov\'asz's neighborhood complex $N(G)$ is homotopy equivalent to $B(G)$ (see \cite{BK1}). Lov\'asz \cite{Lovasz} show the inequality $\chi (G) \geq \conn(N(G)) + 3$ for every graph $G$. Here for a space $X$, $\conn(X)$ denotes the largest integer $n$ such that $X$ is $n$-connected. Hell noted that if this topological lower bound of graphs $G$ and $H$ are tight, then $\chi(G \times H) = \min \{ \chi(G), \chi(H)\}$ holds (see the end of \cite{Hell}). However, this lower bound can be arbitrarily bad (see \cite{Matsushita 1} or Section 12 of \cite{Walker}).
\end{rem}

Next we recall Csorba's functor $A : \SS^{\ZZ_2} \rightarrow \GG$. Let $K$ be a $\ZZ_2$-simplicial complex with involution $\alpha$. The vertex set of $A(K)$ coincides with the vertex set of $K$. Two vertices $v$ and $w$ are adjacent if and only if $\{ \alpha(v), w\}$ is a simplex of $K$. Note that $v$ and $\alpha(v)$ are adjacent in $A(K)$ for every vertex $v$.

We are now ready to prove Theorem \ref{thm 1.2}. For the reader's convenience, we state it again:

\vspace{2mm} \noindent
{\bf Theorem 1.5.}
{\it Csorba's functor $A : \SS^{\ZZ_2} \rightarrow \GG$ is a left adjoint functor preserving finite limits. Moreover, the geometric realization of its right adjoint is naturally $\ZZ_2$-homotopy equivalent to the box complex functor.}

\vspace{2mm}
This theorem is a combination of Proposition \ref{prop 3.1} and Theorem \ref{thm 3.3} mentioned below.

\begin{prop} \label{prop 3.1}
The functor $A: \SS^{\ZZ_2} \rightarrow \GG$ preserves finite limits. In particular, $A$ preserves finite products and hence $A(K \times L)$ and $A(K) \times A(L)$ are isomorphic.
\end{prop}
\begin{proof}
Let $f,g : K \rightarrow L$ be $\ZZ_2$-simplicial maps. Let $\mathcal{E}(f,g)$ (and $\mathcal{E}(A(f), A(g))$) be the equalizer of $f$ and $g$ ($A(f)$ and $A(g)$, respectively). The vertex sets of $V(A \mathcal{E}(f,g))$ and $\mathcal{E}(A(f), A(g))$ coincide with the set $V(\mathcal{E}(f,g)) = \{ v \in V(K) \; | \; f(v) = g(v)\}$. For vertices $v, w \in V(\mathcal{E}(f,g))$, we have
\begin{eqnarray*}
(v,w) \in E \big( A \mathcal{E}(f,g) \big) &\Leftrightarrow & \{ \alpha(v),w\} \in \mathcal{E}(f,g) \; \Leftrightarrow \; \{ \alpha(v), w\} \in K\\
&\Leftrightarrow & (v,w) \in E\big(A(K) \big) \; \Leftrightarrow \; (v,w) \in E \big( \mathcal{E}(A(f), A(g)) \big).
\end{eqnarray*}
Thus the functor $A$ preserves finite equalizers.

Next we shall show that the functor $A$ preserves finite products. Note that the vertex sets of $A(K \times L)$ and $A(K) \times A(L)$ coincide with $V(K) \times V(L)$. Let $\alpha$ and $\beta$ be the involutions on $K$ and $L$, respectively. For a pair $(v,w)$ and $(v', w')$ of elements of $V(K) \times V(L)$, we have
\begin{eqnarray*}
\big( (v,w), (v',w')\big) \in E\big( A(K \times L) \big)
& \Leftrightarrow & \big\{ (\alpha(v),\beta(w)), (v',w')\big\} \in K \times L\\
& \Leftrightarrow & \big\{ \alpha(v), v'\big\} \in K,\; \{ \beta(w), w'\} \in L\\
& \Leftrightarrow & (v,v') \in E(A (K)), \; (w,w') \in E(A(L))\\
& \Leftrightarrow & \big( (v,w), (v',w') \big) \in E \big( A(K) \times A (L)\big).
\end{eqnarray*}
This completes the proof.
\end{proof}

We now prove the adjoint relation concerning the box complex functor $B$.

\begin{thm} \label{thm 3.3}
The pair $(A,B)$ of functors is an adjoint pair from $\SS^{\ZZ_2}$ to $\GG$. Therefore there exsits a graph homomorphism $A(K) \to G$ if and only if there is a $\ZZ_2$-simplicial map $K \to B(G)$.
\end{thm}
\begin{proof}
We construct the precise natural isomorphisms
$$\Phi : \GG(A(K), G) \longrightarrow \SS^{\ZZ_2}(K, B(G)), \; \Psi : \SS^{\ZZ_2}(K, B(G)) \longrightarrow \GG(A(K), G)$$
such that $\Psi$ is the inverse of $\Phi$.

Let $f: A(K) \rightarrow G$ be a graph homomorphism. We define $\Phi(f)(v) \in V(G^{K_2})$ by
$$\Phi(f)(v)(i) = \begin{cases}
f(v) & (i = 1)\\
f(\alpha(v)) & (i = 2).
\end{cases}$$
In other words, $\Phi(f) (v)$ is the edge $(f(v), f(\alpha(v)))$ of $G$ when we identify $V(B(G))$ with $E(G)$. We show that for a simplex $\sigma$ of $K$, the set $\Phi(f)(\sigma) \subset V(G^{K_2})$ is a clique of $G^{K_2}$. This shows that $\Phi(f)(v)$ is looped in $G^{K_2}$ for every $v \in V(K)$ and $\Phi(f)$ is a simplicial map from $K$ to $B(G)$. Let $\{ x_0, \cdots, x_d\}$ be a simplex of $K$. For every pair $i$ and $j$, we have $\{ x_i, x_j\}$ is a simplex of $K$, in other words, $x_i \sim \alpha(x_j)$ and $\alpha(x_i) \sim x_j$ in $A(K)$. Since $f: A(K) \rightarrow G$ is a graph homomorphism, we have
$$\Phi(f)(x_i)(1) = f(x_i) \sim f(\alpha(x_j)) = \Phi(f)(x_j)(2),$$
$$\Phi(f)(x_i)(2) = f(\alpha(x_i)) \sim f(x_j) = \Phi(f)(x_j)(1).$$
Hence $\Phi(f)(x_i)$ and $\Phi(f)(x_j)$ are adjacent in $G^{K_2}$. Therefore we have that $\Phi(f) (\{ x_0, \cdots, x_d\})$ is a clique in $G^{K_2}$.

Next we show that $\Phi(f)$ is $\ZZ_2$-equivariant. First, we see the following equation
$$\Phi(f)(\alpha(v))(1) = f(\alpha(v)) = \Phi(f)(v)(2) = \Phi (f)(v)(\alpha_{K_2}(1)) = (\alpha_{G^{K_2}} (\Phi(f)(v)))(1).$$
Here $\alpha_{K_2}$ and $\alpha_{G^{K_2}}$ are the involutions of $K_2$ and $G^{K_2}$, respectively. Similarly, we can show $\Phi(f)(\alpha(v))(2) = \alpha_{G^{K_2}}(\Phi(f)(v))(2)$. Therefore we have $\Phi(f)(\alpha(v)) = \alpha_{G^{K_2}} (\Phi(f)(v))$. This means that $\Phi(f)$ is $\ZZ_2$-equivariant.

Next we construct $\Psi : \SS^{\ZZ_2}(K, B(G)) \rightarrow \GG(A(K), G)$. For a $\ZZ_2$-simplicial map $g: K \rightarrow B(G)$, define
$$\Psi(g)(v) = g(v)(1)$$
for every $v \in V(K)$. We must show that $\Psi(g) : V(A(K)) \rightarrow V(G)$ is a graph homomorphism. Let $v$ and $w$ be adjacent vertices in $A(K)$. This means $\{ \alpha(v), w\} \in K$ and hence we have $\{ g(\alpha(v)), g(w)\} \in B(G)$. Therefore we have
$$\Psi(g)(v) = g(v)(1) = \alpha_{G^{K_2}}^2 (g(v))(1) = g(\alpha(v))(2) \sim g(w)(1) = \Psi(g)(w).$$
This means that $\Psi(g)$ is a graph homomorphism.

Finally, we check that $\Psi$ is the inverse of $\Phi$. For a graph homomorphism $f: A (K) \rightarrow G$, we have
$$\Psi(\Phi(f))(v) = \Phi(f)(v)(1) = f(v)$$
for every $v \in V(A(K))$, and hence we have $\Psi \circ \Phi = {\rm id}$. On the other hand, for a $\ZZ_2$-simplicial complex $g: K \rightarrow B(G)$, we have
$$\Phi(\Psi(g))(v)(1) = \Psi(g)(v) = g(v)(1),$$
$$\Phi(\Psi(g))(v)(2) = \Psi(g)(\alpha(v)) = g(\alpha(v))(1) = (\alpha_{G^{K_2}} g(v))(1) = g(v)(2)$$
and hence $\Phi \circ \Psi ={\rm id}$. This completes the proof.
\end{proof}

\begin{rem}
Theorem \ref{thm 3.3} implies that there is always a $\ZZ_2$-map $K \to BA(K)$. One can easily show that this map is not an $\ZZ_2$-homotopy equivalence in general. In fact, if $K$ is a 4-cycle with the antipodal action, then $A(K) = K_4$ and hence $BA(K_4) \simeq_{\ZZ_2} S^2$. However, Csorba \cite{Csorba} showed that $BA(\Sd(K))$ and $K$ are $\ZZ_2$-homotopy equivalent.
\end{rem}

It is known that some formulations of box complexes behave like a right adjoint functor (see Theorem 1.6 in Dochtermann-Schultz \cite{DS}). The author noticed that using simplicial sets, we can construct a box complex functor having a left adjoint (see \cite{Matsushita 0}, Corollary 3.2 and Proposition 3.3 of \cite{Matsushita}), and introduce a model structure on the category of graphs. However, this left adjoint functor in \cite{Matsushita} does not preserve finite products although its definition is quite similar to our functor $A$.

\section{Products of $\ZZ_2$-simplicial complexes}

In this section, we study the products of $\ZZ_2$-simplicial complexes. We start with the following proposition. This proposition is probably known to experts, but we give its precise proof for the reader's convenience.


\begin{prop}\label{prop 4.1}
For every pair $K$ and $L$ of $\ZZ_2$-simplicial complexes, the natural map $p : |K \times L| \rightarrow |K| \times |L|$ (see the proof for the definition) is a $\ZZ_2$-homotopy equivalence.
\end{prop}
\begin{proof}
Let $p_1 : K \times L \rightarrow K$ and $p_2 : K \times L \rightarrow L$ be the projections. The natural map $p$ is given by $(|p_1|, |p_2|) : |K \times L| \rightarrow |K| \times |L|$. It suffices to show that $p : F(K \times L) \rightarrow FK \times FL, \; \sigma \mapsto (p_1(\sigma), p_2(\sigma))$ is a $\ZZ_2$-homotopy equivalence. Define $i : FK \times FL \rightarrow F(K \times L)$ by $i(\sigma, \tau) = \sigma \times \tau$. Then we have $pi = {\rm id}_{FK \times FL}$ and $ip \geq {\rm id}_{F(K \times L)}$, and $i$ and $p$ are $\ZZ_2$-equivariant. Hence $p$ is a $\ZZ_2$-homotopy equivalence (see Theorem 4 of \cite{Csorba2} or the proof of Proposition 14.3.10 of \cite{Hirschhorn}).
\end{proof}

Next we show the simplicial approximation theorem for products of $\ZZ_2$-simplicial complexes (Proposition \ref{prop approximation}). First we recall some terminology and notation concerning simplicial complexes from Section 2.C of \cite{Hatcher}. Let $K$ be a simplicial complex. For a simplex $\sigma$ of $K$, the {\it star of $\sigma$} is the subcomplex
$$\{ \tau \in K \; | \; \sigma \cup \tau \in K \},$$
of $K$. The {\it closed star of $\sigma$} is the geometric realization of the star of $\sigma$, and is denoted by $\St_K (\sigma)$. The {\rm open star of $\sigma$ in $\K$} is the union of the interior of $|\tau|$, where $\tau$ runs all simplices of the star of $\sigma$, and denoted by $\st_K(\sigma)$.

We are now ready to prove the $\ZZ_2$-simplicial approximation theorem of products.

\begin{prop} \label{prop approximation}
Let $K$, $L$, and $X$ be $\ZZ_2$-simplicial complexes, and $f : |\K| \times |\L| \rightarrow |\X|$ a $\ZZ_2$-map. Then there is an integer $k \geq 0$ and a simplicial map $\varphi : \Sd^k(K) \times \Sd^k(L) \rightarrow \X$ making the following diagram commute up to $\ZZ_2$-homotopy:

$$\xymatrix{
|\Sd^k(\K) \times \Sd^k(\L)| \ar[r]^{p \;\;} \ar[rrd]_{|\varphi| \; } & |\Sd^k(\K)| \times |\Sd^k(\L)| \ar[r]^{\;\;\;\;\;\;\;\;\;\; \cong} & |\K| \times |\L| \ar[d]^f\\
& & |\X|
}$$
\end{prop}
\begin{proof}
Take $k$ to be sufficiently large so that for every $v \in V(\K)$ and $w \in V(\L)$, $f(\St_{\Sd^k(K)}(v) \times \St_{\Sd^k(L)}(w))$ is contained in some open star of a vertex of $\X$. For simplicity of the notation, we set $K' = \Sd^k(K)$ and $L' = \Sd^k(L)$.

Note that the natural map $p : |K' \times L'| \rightarrow |K'| \times |L'|$ satisfies the following property: For every vertex $(v,w) \in V(K' \times L')$, we have $p(\St_{K' \times L'}(v,w)) \subset \St_{K'}(v) \times \St_{L'}(w)$. Thus for every vertex $(v,w)$ of $K' \times L'$, we have that $f \circ |p|(\St_{K' \times L'}(v,w))$ is contained in some open star of a vertex of $\X$. Let $\varphi : V(K' \times L') \rightarrow V(X)$ be a $\ZZ_2$-set map satisfying $|\varphi|(\St_{K' \times L'}(v,w)) \subset \st_{X}(\varphi(v,w))$ for every $(v,w) \in V(K' \times L')$.

We show that $\varphi$ is a simplicial map. Let $\sigma = \{ (v_0,w_0), \cdots, (v_d, w_d)\}$ be a simplex of $K' \times L'$. Let $x$ be an interior point of $\sigma$, i.e. $x$ is an element of $\st_{K' \times L'}(v_0,w_0) \cap \cdots \cap \st_{K' \times L'}(v_d,w_d)$ by Lemma 2C.2 of \cite{Hatcher}.

Since $f\circ p(\st_{K' \times L'} (v_i,w_i)) \subset \st_{X} (\varphi(v_i,w_i))$, we have that $f \circ p(x)$ is an element of $\st(\varphi(v_0,w_0)) \cap \cdots \cap \st(\varphi(v_d,w_d))$. Thus it again follows from Lemma 2C.2 of \cite{Hatcher} that $\{ \varphi(v_0,w_0), \cdots, \varphi(v_d,w_d)\}$ is a simplex of $X$. Thus $\varphi$ is a simplicial map.

It is clear that for every $x \in |\Sd^k(K) \times \Sd^k(L)|$, $f \circ p(x)$ and $|\varphi|(x)$ are contained in some simplex. Hence the family of maps
$$(1-t) f \circ p + t |\varphi |, \; (t \in [0,1])$$ gives a homotopy from $f \circ p$ to $|\varphi|$. This completes the proof.
\end{proof}

\begin{cor} \label{cor approximation}
For a pair $K$ and $L$ of finite $\ZZ_2$-simplicial complexes and for a graph $G$, there is a $\ZZ_2$-continuous map from $|K| \times |L|$ to $|B(G)|$ if and only if there is a graph homomorphism from $A(\Sd^k (K) \times \Sd^k (L)) = A (\Sd^k (K)) \times A (\Sd^k (L))$ to $G$ for sufficiently large $k$.
\end{cor}
\begin{proof}
Since $A$ preserves finite products (Proposition \ref{prop 3.1}), we have $A(\Sd^k(K) \times \Sd^k(L)) = A(\Sd^k(K)) \times A(\Sd^k(L))$. Then this corollary follows from the adjoint property of $B$ (Theorem \ref{thm 3.3}) and Propostion \ref{prop approximation}.
\end{proof}

Taking $L$ to be a point, we have the following corollary. Of course, this corollary clearly follows from Theorem \ref{thm 1.3} and $\ZZ_2$-simplicial approximation theorem:

\begin{cor} \label{cor approximation 2}
For a $\ZZ_2$-simplicial complex $K$ and for a graph $G$, there is a $\ZZ_2$-continuous map from $|K|$ to $|B(G)|$ if and only if there is a graph homomorphism $A(\Sd^k(K))$ to $G$ for sufficiently large $k$.
\end{cor}

\section{Proof of Theorem \ref{thm 1.2}}

The purpose of this section is to complete the proof of Theorem \ref{thm 1.2}.

\begin{prop} \label{prop 3.5}
Let $K$ and $L$ be $\ZZ_2$-simplicial complexes. Then for sufficiently large $k$, we have
$$\ind (|K| \times |L|) + 2 = \chi \big( A ( \Sd^k(K) \times \Sd^k(L)) \big) = \chi \big( A \circ \Sd^k (K) \times A \circ \Sd^k(L) \big).$$
\end{prop}
\begin{proof}
Since $A$ preserves finite products (Proposition \ref{prop 3.1}), we have
$$\chi(A(\Sd^k(K) \times \Sd^k(L))) = \chi(A \circ \Sd^k(K) \times A \circ \Sd^k(L)).$$

Let $n$ be the $\ZZ_2$-index of $|K| \times |L|$. Then there is a $\ZZ_2$-map from $|K| \times |L|$ to $S^n = |B(K_{n+2})|$. Then Corollary \ref{cor approximation} implies that there is a graph homomorphism $A \circ \Sd^k(K) \times A \circ \Sd^k(L)$ to $K_{n+2}$. This means $\chi(A \circ \Sd^k(K) \times A \circ \Sd^k(L)) \le n+2$. On the other hand, if there is a graph homomorphism $A(\Sd^k(K) \times \Sd^k(L)) \to K_{n+1}$, there is a $\ZZ_2$-map $\Sd^k(K) \times \Sd^k(L)$ to $B(K_{n+1}) = S^{n-1}$. Since $|\Sd^k(K) \times \Sd^k(L)| \simeq_{\ZZ_2} |K| \times |L|$ and $n = \ind(|K| \times |L|)$, this is a contradiction.
\end{proof}

Taking $L$ to be a point, we have the following corollary:

\begin{cor} \label{cor 3.4}
Let $K$ be a $\ZZ_2$-simplicial complex. Then for sufficiently large $k$, we have
$$\ind \big( |K| \big) + 2 = \chi \big(A \circ \Sd^k(K) \big).$$
\end{cor}

We are now ready to prove Theorem \ref{thm 1.2}.

\vspace{2mm} \noindent
{\it Proof of Theorem \ref{thm 1.2}.}
Suppose that $\HH(n)$ is true (see Section 1 for the definition) and suppose that $\ind(X \times Y) = n-2$. Let $K$ and $L$ be finite $\ZZ_2$-simplicial complexes such that $|K| \simeq_{\ZZ_2} X$ and $|L| \simeq_{\ZZ_2} Y$. Then we have
\begin{eqnarray*}
\min \big\{ \ind \big(|K| \big), \ind \big(|L| \big) \big\} + 2 
&=& \lim_{k \rightarrow \infty} \min \big\{ \chi \big(A \circ \Sd^k(K)\big), \chi \big( A \circ \Sd^k(L) \big) \big\} \\
&=& \lim_{k\rightarrow \infty} \chi \big( (A \circ \Sd^k(K)) \times (A \circ \Sd^k(L))\big)\\
&=& \ind\big( |K| \times |L| \big) + 2.
\end{eqnarray*}
The first, second, and third equalities follow from Corollary \ref{cor 3.4}, $\HH(n)$, and Proposition \ref{prop 3.5}, respectively.
\qed

\section{Direct proof of Corollary \ref{cor 1.2.1}}

As is stated in Section 1, Corollary \ref{cor 1.2.1} follows from Theorem \ref{thm 1.2} and \cite{ES}, The purpose of this section is to give a direct proof of Corollary \ref{cor 1.2.1}. In fact, it is not necessary to assume that $\ZZ_2$-complexes $X$ and $Y$ in Corollary \ref{cor 1.2.1} are finite. 

Before giving the proof, we introduce some terminology. For a free $\ZZ_2$-complex $X$, we write $\overline{X}$ to indicate the orbit space of $X$. Suppose that $X$ is connected. Then $\pi_1(X)$ is a subgroup of $\pi_1 ( \overline{X} )$ with index $2$. We call an element $x$ of $\pi_1 ( \overline{X} )$ {\it even} if $x$ belongs to $\pi_1(X)$, and {\it odd} if not. The following lemma is easily deduced from covering space theory (see Proposition 1.33 of \cite{Hatcher}), so we omit the proof.

\begin{lem} \label{lem 6.1}
Let $X$ be a connected free $\ZZ_2$-complex. Then there is a $\ZZ_2$-map from $X$ to $S^1$ if and only if there is a group homomorphism $f : \pi_1 (\overline{X} ) \rightarrow \ZZ \; (\cong \pi_1(S^1))$ satisfying $f^{-1}(2\ZZ) = \pi_1(X)$.
\end{lem}

We now turn to the proof of Corollary \ref{cor 1.2.1}. Note that $\pi_1 ( \overline{X \times Y} )$ can be regarded as a subgroup of $\pi_1 ( \overline{X} \times \overline{Y} ) = \pi_1 ( \overline{X} ) \times \pi_1 (\overline{Y} )$ consisting of pairs $(x,y)$ such that the parities of $x \in \pi_1 ( \overline{X} )$ and $y \in \pi_1 ( \overline{Y} )$ coincide. For a pair $x \in \pi_1 ( \overline{X} )$ and $y \in \pi_1 ( \overline{Y} )$ having the same parities, we write $(x,y)$ to indicate the associated element of $\pi_1 ( \overline{X \times Y} )$. Then $(x,y) \in \pi_1 ( \overline{X \times Y} )$ is odd if and only if $x$ is odd in $\pi_1 ( \overline{X} )$ (and hence $y$ is odd in $\pi_1 ( \overline{Y} )$).

Suppose that there is a $\ZZ_2$-map from $X \times Y$ to $S^1$. It follows from Lemma \ref{lem 6.1} that there is a group homomorphism $f : \pi_1 ( \overline{X \times Y}) \rightarrow \ZZ$ with $f^{-1}(2\ZZ) = \pi_1 (X \times Y )$. Let $(x_0, y_0)$ be an element of $\pi_1 ( \overline{X \times Y} )$ such that $f(x_0, y_0) = 1 \in \ZZ$. Then we have $f(x_0^2, y_0^2) = 2$. Since $f(x_0^2, 0) \in 2\ZZ$ and $f(0, y_0^2) \in 2\ZZ$, we have that one of $f(x_0^2,0) / 2$ and $f(0,y_0^2) / 2$ is odd. Without loss of generality, we can assume that $f(x_0^2, 0 ) / 2$ is odd. Then define the homomorphism $g: \pi_1 (\overline {X}) \rightarrow \ZZ$ by
$$g(x) = \frac{f(x^2,0)}{2}.$$
It suffices to show that $g$ is a group homomorphism and $g^{-1}(2\ZZ) = \pi_1(X)$.

Regard $\pi_1(\overline{X})$ as a subgroup of $\pi_1(\overline{X}) \times \pi_1(\overline{Y})$. From this viewpoint, the commutator subgroup of $\pi_1 ( \overline{X} )$ is contained in the commutator subgroup of $\pi_1 ( \overline{X \times Y} )$. In fact, for odd elements $x_1, x_2 \in \pi_1 ( \overline{X})$, we have
$$([x_1, x_2], 0) = [(x_1,y_0), (x_2,y_0)].$$
Other cases are similarly obtained. Therefore, for elements $x_1, x_2 \in \pi_1 \big( \overline{X} \big)$, we have $f(x_1x_2,0) = f(x_2x_1,0)$. Therefore we have
$$g(x_1 x_2) = \frac{f(x_1x_2 x_1 x_2,0)}{2} = \frac{f(x_1^2 x_2^2,0)}{2} = \frac{f(x_1^2,0)}{2} + \frac{f(x_2^2,0)}{2} = g(x_1) + g(x_2),$$
and hence $g$ is a group homomorphism.

Finally, we show $g^{-1}(2\ZZ) = \pi_1(X)$. It is clear that $g(x)$ is even for every even element $x$, and $g(x_0)$ is odd. For every odd element $x \in \pi_1 (\overline{X} )$, we have that $g(x) + g(x_0) = g(xx_0)$ is even and hence $g(x)$ is odd. This completes the proof.

\end{document}